\documentclass[preprint,12pt]{elsarticle}

\usepackage{amssymb,amsmath,amsthm,latexsym}
\usepackage{graphicx}
\usepackage{epsfig}
\usepackage{epstopdf}

\newcommand{\sysn}{\left\{\begin{array}{rcl}}
\newcommand{\sysk}{\end{array}\right.}

\newcommand{\ingrw}[2]{\includegraphics[width=#1mm]{#2}}

\newtheorem{theorem}{Theorem}[section]

\theoremstyle{example}

\newtheorem{proposition}[theorem]{Proposition}
\theoremstyle{definition}
\newtheorem{definition}[theorem]{Definition}

\newtheorem{corollary}[theorem]{Corollary}

\journal{Topology and its Applications}

\begin{document}

\title{Projective versions of the properties \\ in the Scheepers Diagram}

\author{Alexander V. Osipov}

\address{Krasovskii Institute of Mathematics and Mechanics, \\ Ural Federal
 University, Ural State University of Economics, Yekaterinburg, Russia}

\ead{OAB@list.ru}

\begin{abstract}
Let $\mathcal{P}$ be a topological property. A.V. Arhangel'skii
calls $X$ {\it projectively $\mathcal{P}$} if every second
countable continuous image of $X$ is $\mathcal{P}$. Lj.D.R.
Ko$\check{c}$inac characterized the classical covering properties
of Menger, Rothberger, Hurewicz and Gerlits-Nagy in term of
continuous images in $\mathbb{R}^{\omega}$.  In this paper we
study the functional characterizations of all projective versions
of the selection properties in the Scheepers Diagram.
\end{abstract}


\maketitle 


\section{Introduction}

Many topological properties are characterized in terms
 of the following classical selection principles.
 Let $\mathcal{A}$ and $\mathcal{B}$ be sets consisting of
families of subsets of an infinite set $X$. Then:

$S_{1}(\mathcal{A},\mathcal{B})$ is the selection hypothesis: for
each sequence $(A_{n}: n\in \mathbb{N})$ of elements of
$\mathcal{A}$ there is a sequence $(b_{n}: n\in\mathbb{N})$ such
that for each $n$, $b_{n}\in A_{n}$, and $\{b_{n}: n\in\mathbb{N}
\}$ is an element of $\mathcal{B}$.

$S_{fin}(\mathcal{A},\mathcal{B})$ is the selection hypothesis:
for each sequence $(A_{n}: n\in \mathbb{N})$ of elements of
$\mathcal{A}$ there is a sequence $(B_{n}: n\in \mathbb{N})$ of
finite sets such that for each $n$, $B_{n}\subseteq A_{n}$, and
$\bigcup_{n\in\mathbb{N}}B_{n}\in\mathcal{B}$.

$U_{fin}(\mathcal{A},\mathcal{B})$ is the selection hypothesis:
whenever $\mathcal{U}_1$, $\mathcal{U}_2, ... \in \mathcal{A}$ and
none contains a finite subcover, there are finite sets
$\mathcal{F}_n\subseteq \mathcal{U}_n$, $n\in \mathbb{N}$, such
that $\{\bigcup \mathcal{F}_n : n\in \mathbb{N}\}\in \mathcal{B}$.

The papers \cite{jmss,ko,sash,sch3,scheep,sch1,tss1,bts,tszd} have
initiated the simultaneous
 consideration of these properties in the case where $\mathcal{A}$ and
 $\mathcal{B}$ are important families of open covers of a
 topological space $X$.

\medskip
In this paper, by a cover we mean a nontrivial one, that is,
$\mathcal{U}$ is a cover of $X$ if $X=\bigcup \mathcal{U}$ and
$X\notin \mathcal{U}$.

 An open cover $\mathcal{U}$ of a space $X$ is:

 $\bullet$ an {\it $\omega$-cover} if every finite subset of $X$ is contained in a
 member of $\mathcal{U}$.

$\bullet$ a {\it $\gamma$-cover} if it is infinite and each $x\in
X$ belongs to all but finitely many elements of $\mathcal{U}$.

\bigskip
For a topological space $X$ we denote:

$\bullet$ $\mathcal{O}$ --- the family of all open covers of $X$;

$\bullet$ $\mathcal{O}^{\omega}_{cz}$ --- the family of all
countable cozero covers of $X$;

$\bullet$ $\Gamma$ --- the family of all open $\gamma$-covers of
$X$;

$\bullet$ $\Gamma_{cz}$ --- the family of all cozero
$\gamma$-covers of $X$;

$\bullet$ $\Omega$ --- the family of all open $\omega$-covers of
$X$;

$\bullet$ $\Omega^{\omega}_{cz}$ --- the family of countable
cozero $\omega$-covers of $X$;

$\bullet$ $\Omega^{\omega}_{cl}$ --- the family of all countable
clopen $\omega$-covers of $X$;

$\bullet$ $\mathcal{D}$ --- the family of all  dense subsets of
$X$;

$\bullet$ $\mathcal{S}$ --- the family of all sequentially dense
subsets of $X$;

$\bullet$ $\mathcal{D}^{\omega}$ --- the family of all countable
dense subsets of $X$;

$\bullet$ $\mathcal{S}^{\omega}$ --- the family of all countable
sequentially dense subsets of $X$.

\bigskip

Many equivalences hold among the selection properties, and the
surviving ones appear in the following the Scheepers Diagram
(where an arrow denotes implication), to which no arrow can be
added except perhaps from $U_{fin}(\mathcal{O}, \Gamma)$ or
$U_{fin}(\mathcal{O}, \Omega)$ to $S_{fin}(\Gamma, \Omega)$
\cite{jmss}.

\begin{center}

\ingrw{90}{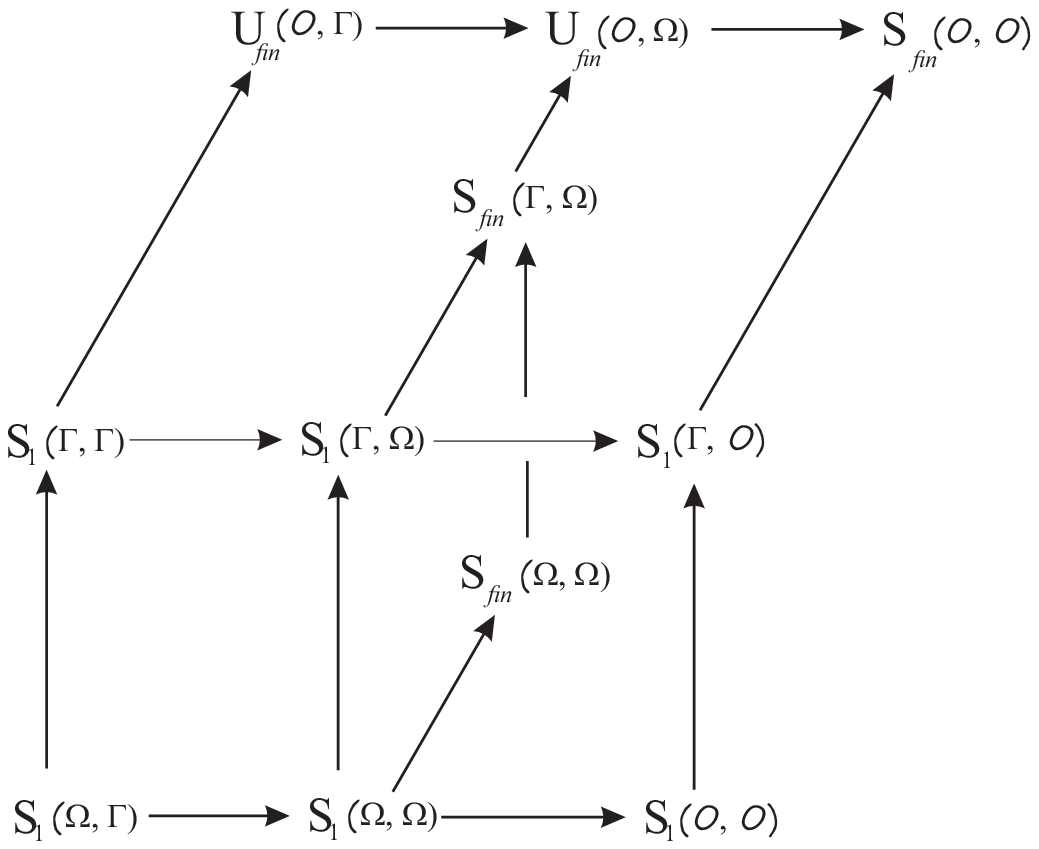}

\medskip

Fig.~1. The Scheepers Diagram for Lindel$\ddot{o}$f spaces.

\end{center}
\bigskip

\medskip
Let $\mathcal{P}$ be a topological property. A.V. Arhangel'skii
calls $X$ {\it projectively $\mathcal{P}$ } if every second
countable continuous image of $X$ is $\mathcal{P}$ \cite{arch3}.

A.V. Arhangel'skii consider projective $\mathcal{P}$ for
$\mathcal{P}=\sigma$-compact, analytic and other properties  in
\cite{arh}. The projective selection principles were introduced
and first time considered in \cite{koc}. Lj.D.R. Ko$\check{c}$inac
characterized the classical covering properties of Menger,
Rothberger, Hurewicz and Gerlits-Nagy in term of continuous images
in $\mathbb{R}^{\omega}$. Characterizations of the classical
covering properties in terms a selection principle restricted to
countable covers by cozero sets are given in \cite{bcm}.

In this paper we study the functional characterizations of the
projective versions of the properties in the Scheepers Diagram
(Fig. 1).

\section{Main definitions and notation}

Let $X$ be a topological space, and $x\in X$. A subset $A$ of $X$
{\it converges} to $x$, $x=\lim A$, if $A$ is infinite, $x\notin
A$, and for each neighborhood $U$ of $x$, $A\setminus U$ is
finite. Consider the following collection:

$\bullet$ $\Omega_x=\{A\subseteq X : x\in \overline{A}\setminus
A\}$;

$\bullet$ $\Gamma_x=\{A\subseteq X : x=\lim A\}$;

$\bullet$ $\Omega^{\omega}_x=\{A\subseteq X : |A|=\aleph_0$ and
$x\in \overline{A}\setminus A\}$;

$\bullet$ $\Gamma^{\omega}_x=\{A\subseteq X : |A|=\aleph_0$ and
$x=\lim A\}$.

\bigskip

We write $\Pi (\mathcal{A}_x, \mathcal{B}_x)$ (resp., $\Pi
(\mathcal{A}, \mathcal{B}_x)$) without specifying $x$, we mean
$(\forall x) \Pi (\mathcal{A}_x, \mathcal{B}_x)$ (resp., $(\forall
x) \Pi (\mathcal{A}, \mathcal{B}_x)$).

Throughout this paper, all spaces are assumed to be Tychonoff. The
set of positive integers is denoted by $\mathbb{N}$. Let
$\mathbb{R}$ be the real line, we put $\mathbb{I}=[0,1]\subset
\mathbb{R}$, and let $\mathbb{Q}$ be the rational numbers. For a
space $X$, we denote by $C_p(X)$ the space of all real-valued
continuous functions on $X$ with the topology of pointwise
convergence. The symbol $\bf{0}$ stands for the constant function
to $0$. Since $C_p(X)$ is homogenous space we may always consider
the point $\bf{0}$ when studying local properties of this space.

A basic open neighborhood of $\bf{0}$ is of the form $[F,
(-\epsilon, \epsilon)]=\{f\in C(X): f(F)\subset (-\epsilon,
\epsilon)\}$, where $F$ is a finite subset of $X$ and
$\epsilon>0$.

 We recall that a subset of $X$ that is the
 complete preimage of zero for a certain function from~$C(X)$ is called a zero-set.
A subset $O\subseteq X$  is called  a cozero-set (or functionally
open) of $X$ if $X\setminus O$ is a zero-set.

Recall that the cardinal $\mathfrak{p}$ is the smallest cardinal
so that there is a collection of $\mathfrak{p}$ many subsets of
the natural numbers with the strong finite intersection property
but no infinite pseudo-intersection. Note that $\omega_1 \leq
\mathfrak{p} \leq \mathfrak{c}$.

For $f,g\in \mathbb{N}^{\mathbb{N}}$, let $f\leq^{*} g$ if
$f(n)\leq g(n)$ for all but finitely many $n$. $\mathfrak{b}$ is
the minimal cardinality of a $\leq^{*}$-unbounded subset of
$\mathbb{N}^{\mathbb{N}}$. A set $B\subset [\mathbb{N}]^{\infty}$
is unbounded if the set of all increasing enumerations of elements
of $B$ is unbounded in $\mathbb{N}^{\mathbb{N}}$, with respect to
$\leq^{*}$. It follows that $|B|\geq \mathfrak{b}$ (See \cite{do}
for more on small cardinals including $\mathfrak{p}$).

\medskip

\begin{theorem}(Noble \cite{nob})\label{th31}  A space $C_{p}(X)$ is separable
if and only if $X$ has a coarser second countable topology.
\end{theorem}

  If $X$ is a space and $A\subseteq X$, then the sequential closure of $A$,
 denoted by $[A]_{seq}$, is the set of all limits of sequences
 from $A$. A set $D\subseteq X$ is said to be sequentially dense
 if $X=[D]_{seq}$. If $D$ is a countable sequentially dense subset
 of $X$ then $X$ call sequentially separable space.

 Call $X$ strongly sequentially separable if $X$ is separable and
 every countable dense subset of $X$ is sequentially dense.
 Clearly, every strongly sequentially separable space is
 sequentially separable, and every sequentially separable space is
 separable.

\begin{definition} A space $X$ has the {\it $V$-property} ($X$ $\models$ $V$), if there
 exists  a condensation (= a continuous bijection) $f: X \mapsto Y$ from a space $X$ on a
 separable metric space $Y$ such that $f(U)$ is an $F_{\sigma}$-set
 of $Y$ for any cozero-set $U$ of $X$.
\end{definition}

\begin{theorem} \label{th38} (Velichko \cite{vel}). A space $C_p(X)$ is
sequentially separable if and only if  $X$ $\models$ $V$.
\end{theorem}

\section{The projectively Rothberger property}

\begin{definition}(\cite{os1}) Let $n\in \mathbb{N}$. A set $A\subseteq C_p(X)$ is called {\it $n$-dense} in $C_p(X)$, if for each $n$-finite set
$\{x_1,...,x_n\}\subset X$ such that $x_i\neq x_j$ for $i\neq j$
and an open sets $W_1,..., W_n$ in $\mathbb{R}$ there is $f\in A$
such that $f(x_i)\in W_i$ for $i\in \overline{1,n}$.

\end{definition}

Obviously, that if $A$ is a $n$-dense set of $C_p(X)$ for each
$n\in \mathbb{N}$ then $A$ is a dense set of $C_p(X)$.

For a space $C_p(X)$ we denote:

$\mathcal{D}[n]$
--- the family of all $n$-dense subsets of $C_p(X)$;

$\mathcal{D}^{\omega}[n]$
--- the family of all countable $n$-dense subsets of $C_p(X)$.

\begin{definition} Let $f\in C(X)$ and $n\in \mathbb{N}$. A set $B\subseteq C_p(X)$ is called {\it $n$-dense at point $f$}, if for each $n$-finite set
$\{x_1,...,x_n\}\subset X$ and $\epsilon>0$ there is $h\in B$ such
that $h(x_i)\in (f(x_i)-\epsilon, f(x_i)+\epsilon)$ for $i\in
\overline{1,n}$.
\end{definition}

Obviously, that if $B$ is a $n$-dense at point $f$ for each $n\in
\mathbb{N}$ then $f\in \overline{B}$.

For a space $C_p(X)$ and $f\in C_p(X)$ we denote:

$\mathcal{D}_{f}[n]$
--- the family of all $n$-dense at point $f$ subsets of $C_p(X)$;

$\mathcal{D}^{\omega}_{f}[n]$
--- the family of all countable $n$-dense at point $f$ subsets of $C_p(X)$.

\medskip

 By Theorem 11.3 in \cite{os4}, we proved the following result
 where the symbol $\bf{0}$ stands for the constant function
to $0$.

\begin{theorem}\label{th143} For a space $X$, the following statements are
equivalent:

\begin{enumerate}

\item  $C_p(X)$ satisfies $S_{1}(\mathcal{D}[1],\mathcal{D}[1])$;

\item $X$ satisfies $S_{1}(\mathcal{O}, \mathcal{O})$
$[$Rothberger property$]$;

\item $C_p(X)$ satisfies $S_{1}(\mathcal{D}_{\bf
0}[1],\mathcal{D}_{\bf 0}[1])$;

\item  $C_p(X)$ satisfies $S_{1}(\mathcal{D}[1],\mathcal{D}_{\bf
0}[1])$;

\item  $C_p(X)$ satisfies $S_{1}(\mathcal{D},\mathcal{D}[1])$.

\end{enumerate}

\end{theorem}

 In (\cite{bcm}, Theorem 37), M. Bonanzinga, F. Cammaroto, M.
 Matveev proved

\begin{theorem}\label{bcm} The following
conditions are equivalent for a space $X$:

\begin{enumerate}

\item $X$ is projectively $S_{1}(\mathcal{O},\mathcal{O})$
$[projectively Rothberger]$;

\item every Lindel$\ddot{o}$f continuous image of $X$ is
Rothberger;

\item for every continuous mapping $f: X \mapsto
\mathbb{R}^{\omega}$, $f(X)$ is Rothberger;

\item for every continuous mapping $f: X \mapsto \mathbb{R}$,
$f(X)$ is Rothberger;

\item  $X$ satisfies
$S_{1}(\mathcal{O}^{\omega}_{cz},\mathcal{O})$.

\end{enumerate}

\end{theorem}

Then, we have the next result.

\begin{theorem}\label{th144} For a space $X$, the following statements are
equivalent:

\begin{enumerate}

\item  $C_p(X)$ satisfies
$S_{1}(\mathcal{D}^{\omega}[1],\mathcal{D}[1])$;

\item $X$ is projectively $S_{1}(\mathcal{O},\mathcal{O})$;

\item $C_p(X)$ satisfies $S_{1}(\mathcal{D}^{\omega}_{\bf
0}[1],\mathcal{D}_{\bf 0}[1])$;

\item  $C_p(X)$ satisfies
$S_{1}(\mathcal{D}^{\omega}[1],\mathcal{D}_{\bf 0}[1])$;

\item  $C_p(X)$ satisfies
$S_{1}(\mathcal{D}^{\omega},\mathcal{D}[1])$.

\end{enumerate}

\end{theorem}

\begin{proof} $(1)\Rightarrow(2)$.  Let $( \mathcal{O}_n : n\in \mathbb{N})$ be a
sequence of countable cozero covers of $X$. Let
$\mathcal{O}_n=\{U^n_i: i\in \mathbb{N}\}$ for every $n\in
\mathbb{N}$, $U^n_i=\bigcup\limits_{k\in \mathbb{N}} F^n_{i,k}$
where $F^n_{i,k}$ is a zero-set of $X$ for any $n,i,k\in
\mathbb{N}$. Renumber the rational numbers $\mathbb{Q}$ as $\{q_k
: k\in \mathbb{N}\}$.

We set $A_n=\{f^n_{i,k}\in C(X): f^n_{i,k}\upharpoonright
(X\setminus U^n_i)=1$ and $f^n_{i,k}\upharpoonright F^n_{i,k}=q_k$
for $U^n_i\in \mathcal{O}_n$ , the zero-set  set $F^n_{i,k}\subset
U^n_i$ and $q_k\in \mathbb{Q}\}$. It is not difficult to see that
each $A_n$ is a countable $1$-dense subset of $C_p(X)$ because
each $\mathcal{O}_n$ is a cover of $X$. By the assumption there
exists $f^n_{i(n),k(n)}\in A_n$ such that $\{f^n_{i(n),k(n)} :
n\in \mathbb{N}\}\in \mathcal{D}^{\omega}[1]$.

 For each $f^n_{i(n),k(n)}$ we
take $U^n_{i(n)}\in \mathcal{O}_n$ such that
$f^n_{i(n),k(n)}\upharpoonright(X\setminus U^n_{i(n)})=1$.

 Set $\mathcal{U}=\{ U^n_{i(n)} : n\in \mathbb{N}\}$. For $x\in X$ we consider the basic open neighborhood
 $[x, W]$ of $\bf{0}$, where $W=(-\frac{1}{2},\frac{1}{2})$.

 Note that there is $m\in \mathbb{N}$ such that
$[x, W]$ contains $f^m_{i(n),k(n)}\in \{f^n_{i(n),k(n)} : n\in
\mathbb{N}\}$. This means $x\in
 U^m_{i(m)}$. Consequently $\mathcal{U}$ is a countable cozero cover
of $X$. By Theorem \ref{bcm}, $X$ is projectively
$S_{1}(\mathcal{O},\mathcal{O})$.

$(2)\Rightarrow(3)$. Let $B_n\in \mathcal{D}^{\omega}_{f}[1]$ for
each $n\in \mathbb{N}$. We renumber $\{B_n\}_{n\in \mathbb{N}}$ as
$\{B_{i,j}\}_{i,j\in \mathbb{N}}$.  Since $C(X)$ is homogeneous,
we may think that $f=\bf{0}$.  We set
$\mathcal{U}_{i,j}=\{g^{-1}[(-1/i, 1/i)] : g\in B_{i,j}\}$ for
each $i,j\in \mathbb{N}$. Since $B_{i,j}\in
\mathcal{D}^{\omega}_{\bf 0}[1]$, $\mathcal{U}_{i,j}$ is a
countable cozero cover of $X$ for each $i,j\in \mathbb{N}$. In
case the set $M=\{i\in \mathbb{N}: X\in \mathcal{U}_{i,j} \}$ is
infinite, choose $g_{m}\in B_{m,j}$ $m\in M$ so that
$g^{-1}[(-1/m, 1/m)]=X$, then $\{g_m : m\in \mathbb{N}\}\in
\mathcal{D}_{\bf 0}[1]$.

So we may assume that there exists $i'\in \mathbb{N}$ such that
for each $i\geq i'$ and $g\in B_{i,j}$ we have that $g^{-1}[(-1/i,
1/i)]$ is not $X$.

For the sequence $\mathcal{V}_i=(\mathcal{U}_{i,j} : j\in
\mathbb{N})$ of cozero covers there exists $f_{i,j}\in B_{i,j}$
such that $\mathcal{U}_i=\{f^{-1}_{i,j}[(-1/i,1/i)]:  j\in
\mathbb{N}\}$ is a cover of $X$.  Let $[x, W]$ be any basic open
neighborhood of $\bf{0}$, where $W=(-\epsilon, \epsilon)$,
$\epsilon>0$. There exists $m\geq i'$ and $j\in \mathbb{N}$  such
that $1/m<\epsilon$ and $x\in f^{-1}_{m,j}[(-1/m, 1/m)]$. This
means $\{f_{i,j}: i,j\in \mathbb{N}\}\in \mathcal{D}^{\omega}_{\bf
0}[1]$.

$(3)\Rightarrow(4)$ is immediate.

$(4)\Rightarrow(1)$. Let $A_n\in \mathcal{D}^{\omega}[1]$ for each
$n\in \mathbb{N}$. We renumber $\{A_n\}_{n\in \mathbb{N}}$ as
$\{A_{i,j}\}_{i,j\in \mathbb{N}}$. Renumber the rational numbers
$\mathbb{Q}$ as $\{q_i : i\in \mathbb{N}\}$.  Fix
$i\in\mathbb{N}$. By the assumption there exists $f_{i,j}\in
A_{i,j}$ such that $\{f_{i,j} : j\in \mathbb{N}\}\in
\mathcal{D}_{q_i}[1]$ where $q_i$  is the constant function to
$q_i$. Then $\{f_{i,j} : i,j\in \mathbb{N}\}\in
\mathcal{D}^{\omega}[1]$.

The remaining implications are proved in the same way as in the
proof of Theorem 11.3 in \cite{os4} by replacing $n$-dense (dense)
subsets of $C_p(X)$ with countable $n$-dense (dense) subsets of
$C_p(X)$.

\end{proof}

\begin{proposition}(Proposition 38 in \cite{bcm})\label{pro1}

\begin{enumerate}

\item  A space is Rothberger iff it is Lindel$\ddot{o}$f and
projectively Rothberger \cite{koc}.

\item Every projectively Rothberger space is zero-dimensional.

\item Every space of cardinality less than $cov(\mathcal{M})$ is
projectively Rothberger.

\item The projectively Rothberger property is preserved by
continuous images, by countably unions, by $C^*$-embedded
zero-sets, and by cozero sets.

\end{enumerate}

\end{proposition}

Note that for a Tychonoff space $X$ always there exists a
countable $1$-dense subset in $C_p(X)$. Namely, let $A=\{f_q\in
C(X) :$ where $f_q(x)=q$ for $\forall x\in X$ and $q\in
\mathbb{Q}\}$.

\begin{theorem}\label{th32} A space $X$ is
Lindel$\ddot{o}$f if and only if each $1$-dense set in $C_p(X)$
contains a countable $1$-dense subset.
\end{theorem}

\begin{proof} $(\Rightarrow).$ Let $B$ be a $1$-dense set in
$C_p(X)$ and let $A$ be a countable $1$-dense in $C_p(X)$. Fix
$m\in \mathbb{N}$. For each $x\in X$ there is $f_{q,m,x}\in A$
such that $f_{q,m,x}(x)\in (-\frac{1}{m}+q, \frac{1}{m}+q)$ where
$q\in \mathbb{Q}$. Fix $q\in \mathbb{Q}$. Consider
$\gamma_{q,m}=\{V_{f_{q,m,x}} :$ $x\in X \}$ where
$V_{f_{q,m,x}}=f^{-1}_{q,m,x}[(-\frac{1}{m}+q, \frac{1}{m}+q)]$
for each $x\in X$. Then $\gamma_{q,m}$ is an open cover of $X$,
hence, there is countable subcover
$\gamma'_{q,m}=\{V_{f_{q,m,x_{i}}}: i\in \mathbb{N} \}\subset
\gamma_{q,m}$ of $X$. Consider $\gamma=\bigcup\limits_{m\in
\mathbb{N}, q\in \mathbb{Q}} \gamma'_{q,m}$.

Claim that $C_q=\{f_{q,m,x_{i}} : i\in \mathbb{N}, m\in
\mathbb{N}\}\in B_q$, i.e. $C$ is a $1$-dense set in the point
$f_q(x)=q$. Let $y\in X$ and $\epsilon>0$. Then there are $m'\in
\mathbb{N}$ and $f_{q,m',x_{i'}}\in C$ such that
$\frac{1}{m'}<\epsilon$ and $f_{q,m',x_{i'}}(y_j)\in
(-\frac{1}{m'}+q,\frac{1}{m'}+q)\subset (-\epsilon+q,\epsilon+q)$.
Define $A=\bigcup\limits_{q\in \mathbb{Q}} C_q$. Clearly, that
$A\subseteq B$ and $A$ is a countable $1$-dense subset of
$C_p(X)$.

$(\Leftarrow)$. Let $\gamma=\{U_{\lambda} : \lambda\in \Lambda\}$
be an open cover of $X$. Consider a set $B=\{f_{x,\lambda}\in C(X)
: f_{x,\lambda}(x)=q$ and $f(X\setminus U_{\lambda})\subset \{0\}$
where $x\in U_{\lambda}$ and $q\in \mathbb{Q}\}$. Since the space
$X$ is Tychonoff and $\gamma$ is an open cover of $X$, $B$ is a
$1$-dense subset of $C_p(X)$. There is a countable $1$-dense
subset $A=\{f_{x_i,\lambda_i}\in C(X) : i\in \mathbb{N} \}\subset
B$. Then $\beta=\{U_{\lambda_i}: i\in \mathbb{N} \}$ is a
countable cover of $X$.
\end{proof}

By Proposition \ref{pro1} and Theorem \ref{th32}, we have the next

\begin{proposition}\label{pr12}

\begin{enumerate}

\item A space $C_p(X)$ has the property
$S_{1}(\mathcal{D}[1],\mathcal{D}[1])$ if and only if it is has
the property $S_{1}(\mathcal{D}^{\omega}[1],\mathcal{D}[1])$ and
each $1$-dense subset of $C_p(X)$  contains a countable $1$-dense
subset of $C_p(X)$.

\item If a space $X$ has cardinality less than $cov(\mathcal{M})$
then $C_p(X)$ has the property
$S_{1}(\mathcal{D}^{\omega}[1],\mathcal{D}[1])$.

\item  If $f:X \rightarrow Y$ is continuous mapping from a
Tychonoff space $X$ onto a Tychonoff space $Y$ and $C_p(X)$ has
the property $S_{1}(\mathcal{D}^{\omega}[1],\mathcal{D}[1])$, then
$C_p(Y)$ has the property
$S_{1}(\mathcal{D}^{\omega}[1],\mathcal{D}[1])$.

\item  If $C_p(X)$ has the property
$S_{1}(\mathcal{D}^{\omega}[1],\mathcal{D}[1])$, then
$C_p(X)^{\omega}$ has the property
$S_{1}(\mathcal{D}^{\omega}[1],\mathcal{D}[1])$.

\item  If $X$ has the property projectively Rothberger and $Y$ is
$C^*$-embedded zero-set in $X$ (or cozero set of $X$), then
$C_p(Y)$ has the property
$S_{1}(\mathcal{D}^{\omega}[1],\mathcal{D}[1])$.

\end{enumerate}

\end{proposition}

By Theorem 40 in \cite{bcm} and Theorem \ref{th144}, we have the
next result.

\begin{proposition} If $C_p(X^n)$ has the property
$S_{1}(\mathcal{D}^{\omega}[1],\mathcal{D}[1])$ for every $n\in
\mathbb{N}$, then all countable subspaces of $C_p(X)$ have
countable strong fan tightness.
\end{proposition}

\section{The projectively Menger property}

 By Theorem 12.1 in \cite{os4}, we have the following result.

\begin{theorem}\label{th1444} For a space $X$, the following statements are
equivalent:

\begin{enumerate}

\item  $C_p(X)$ satisfies
$S_{fin}(\mathcal{D}[1],\mathcal{D}[1])$;

\item $X$ satisfies $S_{fin}(\mathcal{O}, \mathcal{O})$ $[$Menger
property$]$;

\item $C_p(X)$ satisfies $S_{fin}(\mathcal{D}_{\bf
0}[1],\mathcal{D}_{\bf 0}[1])$;

\item  $C_p(X)$ satisfies $S_{fin}(\mathcal{D}[1],\mathcal{D}_{\bf
0}[1])$;

\item  $C_p(X)$ satisfies $S_{fin}(\mathcal{D},\mathcal{D}[1])$.

\end{enumerate}

\end{theorem}

 In (\cite{bcm}, Theorem 6), M. Bonanzinga, F. Cammaroto, M.
 Matveev proved

\begin{theorem} The following
conditions are equivalent for a space $X$:

\begin{enumerate}

\item $X$ is projectively $S_{fin}(\mathcal{O},\mathcal{O})$
$[projectively Menger]$;

\item every Lindel$\ddot{o}$f continuous image of $X$ is Menger;

\item for every continuous mapping $f: X \mapsto
\mathbb{R}^{\omega}$, $f(X)$ is Menger;

\item for every continuous mapping $f: X \mapsto
\mathbb{R}^{\omega}$, $f(X)$ is not dominating;

\item  $X$ satisfies
$S_{fin}(\mathcal{O}^{\omega}_{cz},\mathcal{O})$.

\end{enumerate}

\end{theorem}

\begin{theorem}\label{th243} For a space $X$, the following statements are
equivalent:

\begin{enumerate}

\item  $C_p(X)$ satisfies
$S_{fin}(\mathcal{D}^{\omega}[1],\mathcal{D}[1])$;

\item $X$ is projectively $S_{fin}(\mathcal{O},\mathcal{O})$;

\item $C_p(X)$ satisfies $S_{fin}(\mathcal{D}^{\omega}_{\bf
0}[1],\mathcal{D}_{\bf 0}[1])$;

\item  $C_p(X)$ satisfies
$S_{fin}(\mathcal{D}^{\omega}[1],\mathcal{D}_{\bf 0}[1])$;

\item  $C_p(X)$ satisfies
$S_{fin}(\mathcal{D}^{\omega},\mathcal{D}[1])$.

\end{enumerate}

\end{theorem}

\begin{proof} Similarly to the proofs of Theorem \ref{th1444}
and Theorem 12.1 in \cite{os4}.
\end{proof}

\begin{proposition}(Proposition 8 in \cite{bcm})\label{pro2}

\begin{enumerate}

\item  A space is Menger if and only if it is Lindel$\ddot{o}$f
and projectively Menger \cite{koc}.

\item Every $\sigma$-pseudocompact space is projectively Menger.

\item Every space of cardinality less than $\mathfrak{d}$ is
projectively Menger.

\item The projectively Menger property is preserved by continuous
images, by countably unions, by $C^*$-embedded zero-sets
(Proposition 14 in \cite{bcm}), and by cozero sets (Proposition 16
in \cite{bcm}).
\end{enumerate}

\end{proposition}

By Proposition \ref{pro2} and Theorem \ref{th243}, we have the
next

\begin{proposition}\label{pr32}

\begin{enumerate}

\item A space $C_p(X)$ has the property
$S_{fin}(\mathcal{D}[1],\mathcal{D}[1])$ iff is has the property
$S_{fin}(\mathcal{D}^{\omega}[1],\mathcal{D}[1])$ and each
$1$-dense set in $C_p(X)$  contains a countable $1$-dense set in
$C_p(X)$.

\item If a space $X$ has cardinality less than $\mathfrak{d}$ then
$C_p(X)$ has the property
$S_{1}(\mathcal{D}^{\omega}[1],\mathcal{D}[1])$.

\item  If $f:X\mapsto Y$ is continuous mapping from a Tychonoff
space $X$ onto a Tychonoff space $Y$ and $C_p(X)$ has the property
$S_{fin}(\mathcal{D}^{\omega}[1],\mathcal{D}[1])$, then $C_p(Y)$
has the property
$S_{fin}(\mathcal{D}^{\omega}[1],\mathcal{D}[1])$.

\item  If $C_p(X)$ has the property
$S_{fin}(\mathcal{D}^{\omega}[1],\mathcal{D}[1])$, then
$C_p(X)^{\omega}$ has the property
$S_{fin}(\mathcal{D}^{\omega}[1],\mathcal{D}[1])$.

\item  If $X$ has the projectively Menger property and $Y$ is
$C^*$-embedded zero-set in $X$ (or cozero set of $X$), then
$C_p(Y)$ has the property
$S_{fin}(\mathcal{D}^{\omega}[1],\mathcal{D}[1])$.

\end{enumerate}

\end{proposition}

By Theorem 18 in \cite{bcm} and Theorem \ref{th243}, we have the
next proposition.

\begin{proposition} If $C_p(X^n)$ has the property
$S_{fin}(\mathcal{D}^{\omega}[1],\mathcal{D}[1])$ for every $n\in
\mathbb{N}$,  then all countable subspaces of $C_p(X)$ have
countable fan tightness.
\end{proposition}

\section{The projectively Hurewicz property}

\begin{definition} (Sakai)
An $\gamma$-cover $\mathcal{U}$ of cozero sets ({\bf F}unctionally
open sets) of $X$ is {\it $\gamma_{F}$-shrinkable} if there exists
a $\gamma$-cover $\{F_U : U\in \mathcal{U}\}$ of zero-sets of $X$
with $F_U\subset U$ for every $U\in \mathcal{U}$.
\end{definition}

For a topological space $X$ we denote:

$\bullet$ $\Gamma_F$ --- the family of all $\gamma_F$-shrinkable
covers of $X$.

By Theorem 4.1 in \cite{os10},  $X$ has the Hurewicz property if
and only if $X$ satisfies $U_{fin}(\Gamma_F, \Gamma)$ and $X$ is
Lindel$\ddot{o}$f.

\begin{definition} A countable set $A\subset C(X)$ is called {\it weakly
sequential dense subset of $C_p(X)$} if $A=\{\mathcal{F}_n :
\mathcal{F}_n\in [A]^{<\omega}$, $n\in \mathbb{N}\}$  and for each
$f\in C(X)$ there is $\{\mathcal{F}_{n_k} : k\in
\mathbb{N}\}\subset A$ such that $\{\min\limits_{h\in
\mathcal{F}_{n_k}} |h-f| : k\in \mathbb{N}\}\in \Gamma_{\bf 0}$.
\end{definition}

Clearly that any countable sequential dense subset of $C_p(X)$ is
weakly sequential dense.

For a topological space $X$ and $f\in C(X)$ we denote:

$\bullet$ $w\mathcal{S}$ --- the family of all countable weakly
sequential dense subset of $C_p(X)$.

$\bullet$ $w\Gamma_{\bf f}=\{A$ : $A=\{\mathcal{F}_n :
\mathcal{F}_n\in [A]^{<\omega}$, $n\in \mathbb{N}\}\subset C(X)$
such that $\{\min\limits_{h\in \mathcal{F}_{n}}|h-f| : n\in
\mathbb{N} \}\in \Gamma_{\bf 0}\}$.

$\bullet$ $w\Omega_{\bf f}=\{A$: $A=\{\mathcal{F}_n :
\mathcal{F}_n\in [A]^{<\omega}$, $n\in \mathbb{N}\}\subset C(X)$
such that $\{\min\limits_{h\in \mathcal{F}_{n}}|h-f| : n\in
\mathbb{N}\}\in \Omega_{\bf 0}\}$.

$\bullet$ $w\mathcal{D}=\{A$: $A=\{\mathcal{F}_n :
\mathcal{F}_n\in [A]^{<\omega}$, $n\in \mathbb{N}\}\subset C(X)$
such that $A\in w\Omega_{\bf g}$ for each $g\in C(X)\}.$

Note that $\mathcal{S}\subset w\mathcal{S}$, $\Gamma_{\bf
0}\subset w\Gamma_{\bf 0}$, $\Omega_{\bf 0}\subset w\Omega_{\bf
0}$ and $\mathcal{D}\subset w\mathcal{D}$.

\medskip

 In (\cite{bcm}, Theorem 30), M. Bonanzinga, F. Cammaroto, M.
 Matveev proved

\begin{theorem}\label{th33} The following
conditions are equivalent for a space $X$:

\begin{enumerate}

\item $X$ is projectively $U_{fin}(\mathcal{O},\Gamma)$
$[projectively Hurewicz]$;

\item Every Lindel$\ddot{o}$f continuous image of $X$ is Hurewicz;

\item for every continuous mapping $f: X \mapsto
\mathbb{R}^{\omega}$, $f(X)$ is Hurewicz;

\item for every continuous mapping $f: X \mapsto
\mathbb{R}^{\omega}$, $f(X)$ is bounded;

\item  $X$ satisfies $U_{fin}(\mathcal{O}^{\omega}_{cz},\Gamma)$.

\end{enumerate}

\end{theorem}

\begin{theorem}\label{th22} A space $X$ is projectively Hurewicz  if and only if $X$ has the property $U_{fin}(\Gamma_F,
\Gamma)$.
\end{theorem}

\begin{proof} Assume that $X$ has the property $U_{fin}(\Gamma_F,
\Gamma)$. We claim that $X$ satisfies
$U_{fin}(\mathcal{O}^{\omega}_{cz},\Gamma)$. Let $(\mathcal{V}_i:
i\in \mathbb{N})$ be a sequence of countable cozero covers of $X$
where $\mathcal{V}_i=\{V^n_i : n\in \mathbb{N}\}$ for each $i\in
\mathbb{N}$. Since $V^n_i$ is a cozero set, we can represent
$V^n_i=\bigcup\limits_{j=1}^{\infty} F^n_{i,j}$ where $F^n_{i,j}$
is a zero-set of $X$ for each $i,j,n \in \mathbb{N}$ and
$F^n_{i,j}\subset F^n_{i,j+1}$ for $j\in \mathbb{N}$. Consider
$\mathcal{S}_i=\{S^n_i:=\bigcup\limits_{p=1}^{n} F^p_{i,n}: n\in
\mathbb{N} \}$ for each $i\in \mathbb{N}$. Note that
$\mathcal{S}_i\in \Gamma_F$. Since $X$ has the property
$U_{fin}(\Gamma_F, \Gamma)$, there are finite sets
$\mathcal{D}_i\subseteq \mathcal{S}_i$, $n\in \mathbb{N}$, such
that $\{\bigcup \mathcal{D}_i : i\in \mathbb{N}\}\in \Gamma$. It
follows that $X$ satisfies
$U_{fin}(\mathcal{O}^{\omega}_{cz},\Gamma)$.

\end{proof}

By Theorem 4.2 in \cite{os10} and Theorem \ref{th22},  we have the
next theorem.

\begin{theorem}\label{th194} For a space $X$, the following statements are
equivalent:

\begin{enumerate}

\item $C_p(X)$ satisfies $S_{fin}(\Gamma_{\bf 0},w\Gamma_{\bf
0})$;

\item $X$ satisfies $U_{fin}(\Gamma_{F}, \Gamma)$;

\item $X$ is projectively Hurewicz.

\end{enumerate}

\end{theorem}

By Theorem 4.5 in \cite{os10} and Theorem \ref{th22} we have the
next theorem.

\begin{theorem}\label{th1534} Assume that $X$ has the $V$-property. Then the following statements are equivalent:

\begin{enumerate}

\item $C_p(X)$ satisfies $S_{fin}(\mathcal{S},w\mathcal{S})$;

\item $X$ satisfies $U_{fin}(\Gamma_F, \Gamma)$;

\item $X$ is projectively Hurewicz;

\item  $C_p(X)$ satisfies $S_{fin}(\Gamma_{\bf 0}, w\Gamma_{\bf
0})$;

\item  $C_p(X)$ satisfies $S_{fin}(\mathcal{S}, w\Gamma_{\bf 0})$.

\end{enumerate}

\end{theorem}

\begin{proposition}(Proposition 31 in \cite{bcm})\label{pro4}

\begin{enumerate}

\item  A space is Hurewicz iff it is Lindel$\ddot{o}$f and
projectively Hurewicz \cite{koc}.

\item Every $\sigma$-pseudocompact space is projectively Hurewicz.

\item Every space of cardinality less than $\mathfrak{b}$ is
projectively Hurewicz.

\item The projectively Hurewicz property is preserved by
continuous images, by countably unions, by $C^*$-embedded
zero-sets, and by cozero sets.

\end{enumerate}

\end{proposition}

By Proposition \ref{pro4} and Theorem \ref{th1534}, we have the
next

\begin{proposition}\label{pr42}

\begin{enumerate}

\item A space $C_p(X)$ has the property
$S_{fin}(\mathcal{D}[1],w\mathcal{S})$ iff it has the property
$S_{fin}(\mathcal{S},w\mathcal{S})$ and each $1$-dense set in
$C_p(X)$ contains a countable $1$-dense set in $C_p(X)$.

\item If a space $X$ has cardinality less than $\mathfrak{b}$ then
$C_p(X)$ has the property $S_{fin}(\mathcal{S},w\mathcal{S})$.

\item  If $f:X\mapsto Y$ is a continuous mapping from a Tychonoff
space $X$ onto a Tychonoff space $Y$ and $C_p(X)$ has the property
$S_{fin}(\mathcal{S},w\mathcal{S})$, then $C_p(Y)$ has the
property $S_{fin}(\mathcal{S},w\mathcal{S})$.

\item  If $C_p(X)$ has the property
$S_{fin}(\mathcal{S},w\mathcal{S})$, then $C_p(X)^{\omega}$ has
the property $S_{fin}(\mathcal{S}^{\omega},w\mathcal{S})$.

\item  If $X$ has the projectively Hurewicz property and $Y$ is a
$C^*$-embedded zero-set in $X$ (or cozero set of $X$), then
$C_p(Y)$ has the property $S_{fin}(\mathcal{S},w\mathcal{S})$.

\end{enumerate}

\end{proposition}

\section{Projectively Hurewicz + projectively Rothberger properties}

\begin{theorem}(Theorem 50 in \cite{bcm}) The following conditions
are equivalent for a space $X$:

\begin{enumerate}

\item $X$ is both projectively Hurewicz and projectively
Rothberger;

\item every Lindel$\ddot{o}$f continuous image of $X$ is both
Hurewicz and Rothberger;

\item for every continuous mapping $f: X \mapsto
\mathbb{R}^{\omega}$, $f(X)$ is both Hurewicz and Rothberger;

\item for every continuous mapping $f: X \mapsto \mathbb{R}$,
$f(X)$ is both Hurewicz and Rothberger;

\item  For every sequence $(\mathcal{U}_n:n\in \mathbb{N})$ of
countable covers of $X$ by cozero sets, one can pick
$U_n\in\mathcal{U}_n$ so that $(U_n : n\in \mathbb{N})$ is
groupable, that is there is a strictly increasing function $f:
\omega \mapsto \omega$ such that for every $x\in X$, $x\in
\bigcup\{U_i: f(n)\leq i< f(n+1)\}$ for all but finitely many $n$.
\end{enumerate}

\end{theorem}

Recall that  $add(\mathcal{M})=\min\{\mathfrak{b},
cov(\mathcal{M})\}$ \cite{mil}.

\begin{proposition}(Proposition 51 in \cite{bcm})\label{pro5}

\begin{enumerate}

\item A space is both Hurewicz and Rothberger iff it is
Lindel$\ddot{o}$f and it is both projectively Hurewicz and
projectively Rothberger  \cite{koc}.

\item Every space of cardinality less than $add(\mathcal{M})$ is
both projectively Hurewicz and projectively Rothberger.

\end{enumerate}

\end{proposition}

By Proposition \ref{pro5}, Theorem \ref{th1534} and Theorem
\ref{th1444}, we have the next result.

\begin{proposition}\label{pr43}

\begin{enumerate}

\item A space $C_p(X)$ has properties
$S_{fin}(\mathcal{D}[1],w\mathcal{S})$ and
$S_1(\mathcal{D}[1],\mathcal{D}[1])$ iff it has properties
$S_{fin}(\mathcal{S},w\mathcal{S})$ and
$S_1(\mathcal{D}^{\omega}[1],\mathcal{D}[1])$ and each $1$-dense
set in $C_p(X)$ contains a countable $1$-dense set in $C_p(X)$.

\item If a space $X$ has cardinality less than $add(\mathcal{M})$
then $C_p(X)$ has properties $S_{fin}(\mathcal{S},w\mathcal{S})$
and $S_1(\mathcal{D}^{\omega}[1],\mathcal{D}[1])$.

\end{enumerate}

\end{proposition}

\section{The projectively Gerlits-Nagy property}

 Gerlits and Nagy \cite{gn} proved

\begin{theorem}\label{th7} For a space $X$, the following statements are
equivalent:

\begin{enumerate}

\item $C_p(X)$ satisfies $S_{1}(\Omega_{\bf 0}, \Gamma_{\bf 0})$;

\item $X$ satisfies $S_{1}(\Omega, \Gamma)$.

\end{enumerate}

\end{theorem}

By Theorem 5.6 in \cite{os2}, we have the next theorem.

\begin{theorem}\label{th1} Let $X$ be a space with a coarser second countable topology. The
following assertions are equivalent:

\begin{enumerate}

\item $C_p(X)$ satisfies $S_{1}(\mathcal{D},\mathcal{S})$;

\item Each dense subspace of $C_p(X)$ contains a countable
sequentially dense set in $C_p(X)$;

\item $X$ satisfies $S_{1}(\Omega, \Gamma)$;

\item $C_p(X)$ satisfies $S_{1}(\Omega_{\bf 0},\Gamma_{\bf 0})$;

\item $C_p(X)$ satisfies $S_{1}(\mathcal{D},\Gamma_{\bf 0})$.

\end{enumerate}

\end{theorem}

 In (\cite{bcm}, Theorem 54), M. Bonanzinga, F. Cammaroto, M.
 Matveev proved

 \begin{theorem}  The following conditions are equivalent for a
space $X$:

\begin{enumerate}

\item $X$ satisfies projective $S_{1}(\Omega,\Gamma)$ $[$
projectively Gerlits-Nagy $]$;

\item every Lindel\"{o}f image of $X$ has property $(\gamma)$;

\item for every continuous mapping $f:X \mapsto
\mathbb{R}^{\omega}$, $f(X)$ satisfies $S_{1}(\Omega, \Gamma)$;

\item for every continuous mapping $f:X \mapsto \mathbb{R}$,
$f(X)$ satisfies $S_{1}(\Omega, \Gamma)$;

\item for every countable $\omega$-cover $\mathcal{U}$ of $X$ by
cozero sets, one can pick $U_n\in \mathcal{U}$ so that every $x\in
X$ is contained in all but finitely many $U_n$;

\item $X$ satisfies $S_{1}(\Omega_{cz}^{\omega},\Gamma)$.

\end{enumerate}

 \end{theorem}

\begin{theorem}  The following conditions are equivalent for a
space $X$:

\begin{enumerate}

\item $X$ is projective $S_{1}(\Omega,\Gamma)$;

\item $C_p(X)$ satisfies $S_{1}(\Omega^{\omega}_{\bf
0},\Gamma_{\bf 0})$.

\end{enumerate}

\end{theorem}

\begin{proof}$(1)\Rightarrow(2)$. By Theorem 63 in
\cite{bcm}.

$(2)\Rightarrow(1)$. Let $(\mathcal{U}_n : n\in \mathbb{N})$ be a
sequence of open $\omega$-covers of $X$. We set $A_n=\{f\in C(X):
f\upharpoonright (X\setminus U)=0$ for some $U\in \mathcal{U}_n
\}$. It is not difficult to see that each $A_n$ is dense in $C(X)$
since each $\mathcal{U}_n$ is an $\omega$-cover of $X$ and $X$ is
Tychonoff. Let $f$ be the constant function to 1. By the
assumption there exist $f_n\in A_n$ such that $f_n \mapsto f$
($n\mapsto \infty$).

 For each $f_n$ we
take $U_n\in \mathcal{U}_n$ such that
$f_n\upharpoonright(X\setminus U_n)=0$.

 Set $\mathcal{U}=\{ U_n : n\in \mathbb{N}\}$. For each finite subset
$\{x_1,...,x_k\}$ of $X$ we consider the basic open neighborhood
of $f$ $[x_1,...,x_k; W,..., W]$, where $W=(0,2)$.

 Note that there is $n'\in \mathbb{N}$ such that
$[x_1,...,x_k; W,..., W]$ contains $f_n$ for $n>n'$. This means
$\{x_1,...,x_k\}\subset U_n$ for $n>n'$. Consequently
$\mathcal{U}$ is an $\gamma$-cover of $X$.

\end{proof}

\begin{theorem}\label{th37} Let $X$ be a space with a coarser second countable topology. The
following assertions are equivalent:

\begin{enumerate}

\item $C_p(X)$ satisfies
$S_{1}(\mathcal{D}^{\omega},\mathcal{S})$;

\item $C_p(X)$ is strongly sequentially separable;

 \item $X$ satisfies
$S_{1}(\Omega^{\omega}_{cz}, \Gamma)$;

\item $C_p(X)$ satisfies $S_{1}(\Omega^{\omega}_{\bf
0},\Gamma_{\bf 0})$;

\item $C_p(X)$ satisfies $S_{1}(\mathcal{D}^{\omega},\Gamma_{\bf
0})$;

\item $X$ is projectively $S_{1}(\Omega, \Gamma)$ $[$ projectively
Gerlits-Nagy $]$.

\end{enumerate}
\end{theorem}

Recall that $l^{*}(X)\leq \aleph_0$ ($X$ is called an
$\epsilon$-space) if all finite powers of $X$ are
Lindel$\ddot{e}$of (or, by Proposition in \cite{gn}, if every
$\omega$-cover of $X$ contains an at most countable
$\omega$-subcover of $X$).

\begin{proposition}(Proposition 55 in \cite{bcm})

\begin{enumerate}

\item  A space has property $S_1(\Omega,\Gamma)$ iff it is an
$\epsilon$-space and projectively $S_1(\Omega,\Gamma)$ \cite{koc}.

\item Every projectively $S_1(\Omega,\Gamma)$ space is
zero-dimensional.

\item Every space of cardinality less than $\mathfrak{p}$ is
projectively $S_1(\Omega,\Gamma)$.

\item The projectively $S_1(\Omega,\Gamma)$ property is preserved
by continuous images.

\end{enumerate}

\end{proposition}

\begin{proposition} Let $X$ be a space with a coarser second countable topology.  A space $X$ is an $\epsilon$-space iff each dense subset of
$C_p(X)$ consists a countable dense subset of $C_p(X)$.
\end{proposition}

\begin{proof} $(\Rightarrow)$. Let $X$ be an $\epsilon$-space, $D$ be a dense subset of $C_p(X)$. By the Noble's Theorem \ref{th31}, there is
a countable dense subset $S=\{s_i: i\in \mathbb{N}\}$ of $C_p(X)$.
By the Arhangel'skii-Pytkeev Theorem in \cite{arh}, $t(C_p(X))\leq
\aleph_0$. For every $s\in S$ there exists $D_s\subseteq D$ such
that $|D_s|=\aleph_0$ and $s\in \overline{D_s}$. A set
$P=\bigcup\limits_{s\in S} D_s$. Then $|P|=\aleph_0$, $P\subseteq
D$ and $\overline{P}=C_p(X)$.

$(\Leftarrow)$. Let $\mathcal{V}$ be a $\omega$-cover of $X$.
Consider a set $A_{V,K}=\{f\in C(X): f(X\setminus V)\subseteq
\{0\}$ and $f(k)=q_k$ where $k\in K$ and $q_k\in \mathbb{Q} \}$
where $V\in \mathcal{V}$, $K\in [X]^{<\omega}$ and $K\subset V$.
Then $A=\bigcup \{A_{V,K} : V\in \mathcal{V}$, $K\in
[X]^{<\omega}$ and $K\subset V\}$ is a dense subset of $C_p(X)$.

\end{proof}

\begin{proposition}

\begin{enumerate}

\item  A space $C_p(X)$ is strongly sequentially dense and
separable iff it is strongly sequentially separable and each dense
subset of $C_p(X)$ consists a countable dense subset of $C_p(X)$.

\item If $C_p(X)$ is strongly sequentially separable, then $X$ is
zero-dimensional.

\item If a space $X$ of cardinality less than $\mathfrak{p}$, then
$C_p(X)$ is strongly sequentially separable.

\item  If $f:X\mapsto Y$ is a continuous mapping from a Tychonoff
space $X$ onto a Tychonoff space $Y$ with a coarser second
countable topology and $C_p(X)$ is strongly sequentially
separable, then $C_p(Y)$ is strongly sequentially separable.

\end{enumerate}

\end{proposition}

\medskip

By Theorem 6.1 in \cite{os1}, we have

\begin{proposition}\label{pr322} $(CH)$ There is a consistent example of  projectively $S_{1}(\Omega, \Gamma)$ space $X$ with a coarser second countable topology
 such that $X$ is not $S_{1}(\Omega, \Gamma)$.
\end{proposition}

\medskip

\begin{proposition}\label{pr33} There is a projectively $S_{1}(\Omega, \Gamma)$ space
$X$ such that $X^2$ is not projectively $S_{1}(\Omega, \Gamma)$.
\end{proposition}

\begin{proof} Example 58 in \cite{bcm}.
\end{proof}

Note that $S_{1}(\Omega, \Gamma)=S_{fin}(\Omega, \Gamma)$ (see
\cite{jmss}). It follows that the projectively $S_{1}(\Omega,
\Gamma)$ property coincides with the projectively $S_{fin}(\Omega,
\Gamma)$ property.

By Theorem 63 in \cite{bcm} and Theorem \ref{th37},

\begin{proposition} If $C_p(X)$ is strongly sequentially separable, then all countable subspaces of $C_p(X)$ are strictly Fr$\acute{e}$chet-Urysohn.
\end{proposition}

\section{The projectively $S_{1}(\Omega, \Omega)$ property}

In \cite{sak} (Lemma, Theorem 1), M. Sakai proved:

\begin{theorem} $(Sakai)$ For each   space $X$ the
following are equivalent.
\begin{enumerate}

\item $C_p(X)$ satisfies $S_{1}(\Omega_{\bf 0},\Omega_{\bf 0})$.

\item $X^n$ satisfies $S_{1}(\mathcal{O},\mathcal{O})$ $(X^n$ has
Rothberger's property $C^{''})$ for each $n\in \mathbb{N}$.

\item $X$ satisfies $S_1(\Omega, \Omega)$.

\end{enumerate}

\end{theorem}

 In (\cite{sch}, Theorem 13) M. Scheepers proved the following
 result.

\begin{theorem}$(Scheepers)$ \label{th21} For each separable metric space $X$, the
following are equivalent:

\begin{enumerate}

\item $C_p(X)$ satisfies  $S_{1}(\mathcal{D},\mathcal{D})$;

\item $X$ satisfies $S_{1}(\Omega, \Omega)$.

\end{enumerate}

\end{theorem}

\medskip

By Theorem 57 in \cite{bbm1}, \cite{sak} and Theorem \ref{th31},
we have

\begin{theorem}\label{th11} Let $X$ be a space with a coarser second countable topology. The
following assertions are equivalent:

\begin{enumerate}

\item $C_p(X)$ satisfies $S_{1}(\mathcal{D},\mathcal{D})$ $[$
$R$-separable $]$;

\item $C_p(X)$ satisfies $S_{1}(\Omega_{\bf 0},\Omega_{\bf 0})$;

\item $C_p(X)$ satisfies $S_{1}(\mathcal{D},\Omega_{\bf 0})$;

\item $X$ satisfies $S_{1}(\Omega, \Omega)$;

\item $X^n$ satisfies $S_{1}(\mathcal{O},\mathcal{O})$ for each
$n\in \mathbb{N}$.

\end{enumerate}

\end{theorem}

\medskip

\begin{proposition}\label{pr1} The following conditions are equivalent for a
  space $X$:

\begin{enumerate}

\item $X$ is projectively $S_{1}(\Omega, \Omega)$;

\item $X$ satisfies $S_{1}(\Omega^{\omega}_{cz}, \Omega)$;

\item for every continuous mapping $f:X \mapsto
\mathbb{R}^{\omega}$, $f(X)$ is $S_{1}(\Omega, \Omega)$.

\item $C_p(X)$ satisfies $S_{1}(\Omega^{\omega}_{\bf
0},\Omega_{\bf 0})$;

\end{enumerate}

\end{proposition}

\medskip

\begin{proof}
$(1)\Rightarrow(2)$. Let $(\mathcal{U}_n: n\in \mathbb{N})$ be a
sequence of countable $\omega$-covers of $X$ by cozero sets. For
every $n\in \mathbb{N}$ and every $U\in \mathcal{U}_n$ fix a
continuous function $f_{U}: X \mapsto \mathbb{R}$ such that
$U=f^{-1}_{U}[\mathbb{R}\setminus \{0\}]$. Put $h=\prod \{f_{U} :
U\in \mathcal{U}_n, n\in \mathbb{N}\}$. Then $h$ is a continuous
mapping from $X$ onto $h(X)\subset\mathbb{R}^{\omega}$, thus by
(1), $h(X)$ satisfies $S_{1}(\Omega, \Omega)$. Since
$(h(\mathcal{U}_n): n\in \mathbb{N})$ be a sequence of open
$\omega$-covers of $h(X)$ we get (2).

$(2)\Rightarrow(3)$. Let $f$ be a continuous mapping $f:X \mapsto
\mathbb{R}^{\omega}$, and let $(\mathcal{U}_n: n\in \mathbb{N})$
be a sequence of open $\omega$-covers of $f(X)$. Since $f(X)$ is
separable metrizable space, there is a refinement $\mathcal{V}_n$
of $\mathcal{U}_n$ that countable $\omega$-cover of $f(X)$ and
consists of cozero sets. Put $\mathcal{O}_n=\{f^{-1}(V): V\in
\mathcal{V}_n \}$. Then $\mathcal{O}_n$ is a countable
$\omega$-cover of $X$ by cozero sets. By (2), there is $H_n\in
\mathcal{O}_n$ such that $\{ H_n : n\in \mathbb{N}\}$ is a
countable $\omega$-cover of $X$ and consists of cozero sets. For
every $n\in \mathbb{N}$ pick $U_{H_n}\in \mathcal{U}_n$ such that
$U_{H_n}\supset f(H_n)$. Put $\mathcal{F}=\{U_{H_n} : n\in
\mathbb{N}\}$. Then $\mathcal{F}$ is an open $\omega$-cover of
$f(X)$. This proves that $f(X)$ satisfies $S_{1}(\Omega, \Omega)$.

$(3)\Rightarrow(1)$ follows from the fact that every second
countable space can be embedded into $\mathbb{R}^{\omega}$.

$(2)\Rightarrow(4)$.  Let $f\in \bigcap\limits_n\overline{A_n}$,
where $A_n$ is a countable subset of $C(X)$. Since $C(X)$ is
homogeneous, we may think that $f$ is the constant function to the
zero. We set $\mathcal{U}_n=\{g^{-1}[(-1/n, 1/n)] : g\in A_n\}$
for each $n\in \mathbb{N}$. For each $n\in \mathbb{N}$ and each
finite subset $\{x_1,..., x_k\}$ of $X$ a neighborhood
$[x_1,...,x_k; W,..., W]$ of $f$, where $W=(-1/n, 1/n)$, contains
some $g\in A_n$. This means that each $\mathcal{U}_n$ is a
countable cozero $\omega$-cover of $X$. In case the set $M=\{n\in
\mathbb{N}: X\in \mathcal{U}_n \}$ is infinite, choose $g_{m}\in
A_m$ $m\in M$ so that $g^{-1}(-1/m, 1/m)=X$, then $g_m \mapsto f$.
So we may assume that there exists $n\in \mathbb{N}$ such that for
each $m\geq n$ and $g\in A_m$ $g^{-1}[(-1/m, 1/m)]$ is not $X$.
For the sequence $\{\mathcal{U}_m : m\geq n\}$ of cozero
$\omega$-covers there exist $f_m\in A_m$ such that
$\mathcal{U}=\{f^{-1}_m[(-1/m,1/m)]: m>n\}$ is a $\omega$-cover of
$X$. Let $[x_1,...,x_k; W,...,W]$ be any basic open neighborhood
of $f$, where $W=(-\epsilon, \epsilon)$, $\epsilon>0$. There
exists $m\geq n$  such that $\{x_1,...,x_k\}\subset
f^{-1}_m[(-1/m, 1/m)]$ and $1/m<\epsilon$. This means $f\in
\overline{\{f_m : m\in \mathbb{N}\}}$.

$(4)\Rightarrow(2)$. Let $\{\mathcal{U}_n : n\in \mathbb{N}\}$ be
a sequence of countable cozero $\omega$-covers of $X$. Let
$\mathcal{U}_n=\{U_{n,m} : m\in\mathbb{N} \}$. Since $U_{n,m}$ is
cozero set, $U_{n,m}=\bigcup\limits_{i\in \mathbb{N}} F_{i}^{n,m}$
where $F_{i}^{n,m}$ is zero set of $X$ and $F_{i}^{n,m}\subset
F_{i+1}^{n,m}$ for each $i\in \mathbb{N}$.

We set $A_n=\{f_{i}^{n,m}\in C(X): f_i^{n,m}\upharpoonright
(X\setminus U_{n,m})=1$  and $f_{i}^{n,m}\upharpoonright
F_{i}^{n,m}=0$ for $ m,i\in\mathbb{N} \}$. It is not difficult to
see that $f_0\in \overline{A_n}$ ($f_0$ is the constant function
to the zero) for each $n\in \mathbb{N}$  since each
$\mathcal{U}_n$ is an $\omega$-cover of $X$ and $X$ is Tychonoff.

By the assumption, there exists $f_{i(n)}^{n,m(n)}\in A_n$ such
that $f_0\in \overline{\{f_{i(n)}^{n,m(n)} : n\in \mathbb{N}\}}$.

 For each $f_{i(n)}^{n,m(n)}$ we
take $U_{n,m(n)}\in \mathcal{U}_n$.  Set $\mathcal{U}=\{
U_{n,m(n)} : n\in \mathbb{N}\}$.

 For each finite subset
$\{x_1,...,x_k\}$ of $X$ we consider the basic open neighborhood
of $f_0$ $[x_1,...,x_k; W,..., W]$, where $W=(-1,1)$.

 Note that there is $n\in \mathbb{N}$ such that
$[x_1,...,x_k; W,..., W]$ contains $f_{i(n)}^{n,m(n)}$. This means
$\{x_1,...,x_k\}\subset U_{n,m(n)}$. Consequently $\mathcal{U}$ is
an $\omega$-cover of $X$.
\end{proof}

\begin{definition} A space $X$ is {\it $R_{\omega}$-separable} if for every sequence $(D_n:
n\in \mathbb{N})$ of countable dense subspaces of $X$ one can pick
$p_n\in D_n$ so that $\{p_n : n\in \mathbb{N}\}$ is dense in $X$,
i.e $X$ satisfies $S_{1}(\mathcal{D}^{\omega},\mathcal{D})$.

\end{definition}

\begin{theorem}\label{th17} For a space $X$ with a coarser second countable topology, the
following are equivalent:

\begin{enumerate}

\item $C_p(X)$ satisfies $S_{1}(\mathcal{D}^{\omega},\mathcal{D})$
$[$ $R_{\omega}$-separable $]$;

\item $X$ satisfies $S_{1}(\Omega^{\omega}_{cz}, \Omega)$;

\item $C_p(X)$ satisfies $S_{1}(\Omega^{\omega}_{\bf
0},\Omega_{\bf 0})$;

\item $C_p(X)$ satisfies $S_{1}(\mathcal{D}^{\omega}_{\bf
0},\Omega_{\bf 0})$;

\item $X$ is  projectively $S_{1}(\Omega, \Omega)$.

\end{enumerate}

\end{theorem}

\begin{proof} $(1)\Rightarrow(2)$. Let $\{\mathcal{U}_n : n\in \mathbb{N}\}$ be a
sequence of countable cozero $\omega$-covers of $X$ and $\{h_j
:j\in \mathbb{N}\}$ be a countable dense subset of $C_p(X)$. Let
$\mathcal{U}_n=\{U_{n,m} : m\in\mathbb{N} \}$. Since $U_{n,m}$ is
cozero set, $U_{n,m}=\bigcup\limits_{i\in \mathbb{N}} F_{i}^{n,m}$
where $F_{i}^{n,m}$ is zero set of $X$ and $F_{i}^{n,m}\subset
F_{i+1}^{n,m}$ for each $i\in \mathbb{N}$.

We set $A_n=\{f_{i}^{n,m}\in C(X): f_i^{n,m}\upharpoonright
(X\setminus U_{n,m})=1$  and $f_{i}^{n,m}\upharpoonright
F_{i}^{n,m}=h_i$ for $ m,i\in\mathbb{N} \}$. It is not difficult
to see that $A_n$ is a countable dense subspace of $C_p(X)$ for
each $n\in \mathbb{N}$ since each $\mathcal{U}_n$ is an
$\omega$-cover of $X$ and $X$ is Tychonoff.

By the assumption there exists $h_{i(n)}^{n,m(n)}\in A_n$ such
that $\{h_{i(n)}^{n,m(n)} : n\in \mathbb{N}\}$ is a dense subset
of $C_p(X)$.

 For each $h_{i(n)}^{n,m(n)}$ we
take $U_{n,m(n)}\in \mathcal{U}_n$.  Set $\mathcal{U}=\{
U_{n,m(n)} : n\in \mathbb{N}\}$.

 For each finite subset
$\{x_1,...,x_k\}$ of $X$ we consider the basic open neighborhood
of ${\bf 0}$ $[x_1,...,x_k; W,..., W]$, where $W=(-1,1)$.

 Note that there is $n\in \mathbb{N}$ such that
$[x_1,...,x_k; W,..., W]$ contains $h_{i(n)}^{n,m(n)}$. This means
$\{x_1,...,x_k\}\subset U_{n,m(n)}$. Consequently $\mathcal{U}$ is
an $\omega$-cover of $X$.

$(2)\Leftrightarrow(3)\Leftrightarrow(5)$. By Proposition
\ref{pr1} and Noble's Theorem \ref{th31}.

$(3)\Rightarrow(4)$ is immediate.

$(4)\Rightarrow(1)$. Let $D=\{d_n: n\in \mathbb{N} \}$ be a
countable dense subspace of $C_p(X)$. Given a sequence of
countable dense subspace of $C_p(X)$, enumerate it as $\{S_{n,m}:
n,m \in \mathbb{N} \}$. For each $n\in \mathbb{N}$, pick
$d_{n,m}\in S_{n,m}$ so that $d_n\in \overline{\{d_{n,m}: m\in
\mathbb{N}\}}$. Then $\{d_{n,m}: m,n\in \mathbb{N}\}$ is dense in
$C_p(X)$.

\end{proof}

By definition of projectively $S_{1}(\Omega, \Omega)$ space, a
separable metrizable  projectively $S_{1}(\Omega, \Omega)$ space
has the property $S_{1}(\Omega, \Omega)$.

\medskip

\begin{proposition}\label{pr122} $(CH)$ There is a consistent example of  projectively $S_{1}(\Omega, \Omega)$ space $X$ with a coarser second countable topology
such that $X$ is not $S_{1}(\Omega, \Omega)$.
\end{proposition}

\begin{proof}

The Brendle's Theorem in \cite{br} shows that there is a set of
reals $Z$ of size  $\mathfrak{c}$  (=$ \aleph_1$)    which has
property $S_1(B_{\Omega}, B_{\Gamma})$. We can certainly assume
that $Z\subset (0,1)$. Let $Y=Z\cup (-Z)$. Let $X$ be a set $Y$
with the topology of Sorgenfrey line. Then the space $X$ such that
$X$ satisfies $S_{1}(\Omega^{\omega}_{cz}, \Gamma)$ and
$iw(X)=\aleph_0$, but $X^2$ is not Lindel$\ddot{o}$f (Theorem 6.1
in \cite{os1}).
  Since $\Gamma^{\omega}_{cz}\subset \Omega^{\omega}_{cz}$, by Proposition $\ref{pr1}$, $X$
projectively $S_{1}(\Omega, \Omega)$. Since the property
$S_{1}(\Omega, \Omega)$  is preserved under taking finite powers
(Theorem 3.4 in \cite{jmss}), $X$ has not property $S_{1}(\Omega,
\Omega)$ because $X^2$ is not Lindel$\ddot{o}$f.

\end{proof}

\begin{proposition}\label{pr13} $(\diamondsuit_{\omega_1})$ There is a consistent example of  projectively $S_{1}(\Omega, \Omega)$ space $X$ with
a coarser second countable topology such that $X$ is not
$S_{1}(\Omega, \Omega)$.
\end{proposition}

\begin{proof} By Theorem 6.2 in \cite{os1}.

\end{proof}

\begin{corollary}($CH$ or  $\diamondsuit_{\omega_1}$) There is a consistent example of space $X$ with a coarser second countable topology such that
$X$ satisfies $S_{1}(\Omega^{\omega}_{cz}, \Omega)$, but $X^2$ is
not $S_{1}(\mathcal{O}^{\omega}_{cz},\mathcal{O})$.
\end{corollary}

Clearly, that a countable $R_{\omega}$-separable space is a
$R$-separable space.

 It is interesting to consider the following  Question (Question 64, \cite{bbm1}):

 Does there exists an $X$ such that
 $C_p(X)$ is not $R$-separable but contains a dense $R$-separable
 subspace ?

Note that D. Repov$\check{s}$ and L. Zdomskyy  showed that there
exists a Tychonoff space $S$ such that $C_p(S)$ is not
$M$-separable, but $C_p(S)$ contains a dense subset which is
$GN$-separable (hence $R$-separable) under
$\mathfrak{p}=\mathfrak{d}$ \cite{rz}. This implies a positive
answer to Question under $\mathfrak{p}=\mathfrak{d}$.

By Proposition \ref{pr122}, we get a positive answer to this
Question under $CH$ or $\diamondsuit_{\omega_1}$.

\begin{corollary}($CH$ or  $\diamondsuit_{\omega_1}$) There is a consistent example of space $X$ with a coarser second countable topology such that $C_p(X)$ is not $R$-separable, but for every
countable dense subspace $M\subset C_p(X)$, $M$ is $R$-separable.
\end{corollary}

\begin{proposition}
Every space of cardinality less than $cov(\mathcal{M})$ is
projectively $S_1(\Omega,\Omega)$.
\end{proposition}

\section{The projectively $S_{fin}(\Omega, \Omega)$ property}

 In (\cite{arh}, Theorem 2.2.2 in \cite{arh2}) A.V. Arhangel'skii  proved the following result

\begin{theorem} $(A.V. Arhangel'skii)$ \label{th46} For a space $X$, the
following are equivalent:

\begin{enumerate}

\item $C_p(X)$ satisfies $S_{fin}(\Omega_{\bf 0}, \Omega_{\bf
0})$;

 \item $(\forall n\in \mathbb{N})$ $X^{n}$ satisfies $S_{fin}(\mathcal{O},
\mathcal{O})$.

\end{enumerate}

\end{theorem}

It is known (see \cite{jmss}) that $X$ satisfies $S_{fin}(\Omega,
\Omega)$ iff $(\forall n\in \mathbb{N})$ $X^{n}$ satisfies
$S_{fin}(\mathcal{O}, \mathcal{O})$.

\medskip

By Theorem 21 in \cite{bbm1} and Theorem 3.9 in \cite{jmss}, we
have a next result.

\begin{theorem} \label{th44} For a space $X$ with a coarser second countable topology the
following are equivalent:

\begin{enumerate}

\item $C_p(X)$ satisfies $S_{fin}(\mathcal{D},\mathcal{D})$;

\item $X$ satisfies $S_{fin}(\Omega, \Omega)$;

\item $(\forall n\in \mathbb{N})$ $X^{n}\in S_{fin}(\mathcal{O},
\mathcal{O})$;

 \item $C_p(X)$ satisfies $S_{fin}(\Omega_{\bf
0}, \Omega_{\bf 0})$;

\item $C_p(X)$ satisfies $S_{fin}(\mathcal{D}, \Omega_{\bf 0})$.

\end{enumerate}

\end{theorem}

\begin{proposition}\label{pr21} The following conditions are equivalent for a
space $X$:

\begin{enumerate}

\item $X$ is projectively $S_{fin}(\Omega, \Omega)$;

\item $X$ satisfies $S_{fin}(\Omega^{\omega}_{cz}, \Omega)$;

\item for every continuous mapping $f:X \mapsto
\mathbb{R}^{\omega}$, $f(X)$ is $S_{fin}(\Omega, \Omega)$;

\item $C_p(X)$ satisfies $S_{fin}(\Omega^{\omega}_{\bf
0},\Omega_{\bf 0})$.
\end{enumerate}

\end{proposition}

\begin{proof} By Theorem 4.3 in \cite{sak2}, each of the conditions
(1),(2),(4) are equivalent to the condition: for any sequence
$\mathcal{U}_n=\{U_{n,m} : m\in \mathbb{N}\}$ $(n\in \mathbb{N})$
of countable $\omega$-covers of $X$ consisting of cozero-sets in
$X$, there is some $\varphi\in \omega^{\omega}$ such that
$\{U_{n,m} : n\in \mathbb{N}, m\leq \varphi(n)\}$ is an
$\omega$-cover of $X$.

$(3)\Rightarrow(1)$ follows from the fact that every second
countable space can be embedded into $\mathbb{R}^{\omega}$.
\end{proof}

\medskip
\begin{definition} A space $X$ is {\it $M_{\omega}$-separable} if for every sequence $(D_n:
n\in \mathbb{N})$ of countable dense subspaces of $X$ one can
select finite $F_n\subset D_n$ so that $\bigcup\{F_n : n\in
\mathbb{N}\}$ is dense in $X$, i.e $X$ satisfies
$S_{fin}(\mathcal{D}^{\omega},\mathcal{D})$.
\end{definition}

\begin{theorem}\label{th27} For a space $X$ with a coarser second countable topology, the
following are equivalent:

\begin{enumerate}

\item $C_p(X)$ satisfies
$S_{fin}(\mathcal{D}^{\omega},\mathcal{D})$;

\item $X$ satisfies $S_{fin}(\Omega^{\omega}_{cz}, \Omega)$;

\item $C_p(X)$ satisfies $S_{fin}(\Omega^{\omega}_{\bf
0},\Omega_{\bf 0})$;

\item $C_p(X)$ satisfies $S_{fin}(\mathcal{D}^{\omega},\Omega_{\bf
0})$;

\item $X$ is projectively $S_{fin}(\Omega, \Omega)$.

\end{enumerate}

\end{theorem}

\begin{proposition}
Every space of cardinality less than $\mathfrak{d}$ is
projectively $S_{fin}(\Omega,\Omega)$.
\end{proposition}

\section{The projectively $S_{1}(\Gamma, \Omega)$ property}

Recall that a set $A$ in a space $X$ is a $Z_{\sigma}$-set in $X$,
if $A=\bigcup\limits_{i=1}^{\infty} F_i$ where $F_i$ is a zero-set
in $X$ for each $i\in \mathbb{N}$. A set $B$ is a
$CZ_{\sigma}$-set in $X$, if $X\setminus B$ is a $Z_{\sigma}$-set
in $X$.

\begin{definition} A space $X$ is called a {\it $z$-space}, if any
$Z_{\sigma}$-set in $X$ is a $CZ_{\sigma}$-set in $X$.
\end{definition}

Note that if a perfectly normal space $X$ is a $z$-space, then $X$
is a $\sigma$-space (every $F_{\sigma}$-set is a
$G_{\delta}$-set).

\begin{theorem}\label{pr80} For a $z$-space $X$, the following statements are
equivalent:

\begin{enumerate}

\item $X$ is projectively $S_{1}(\Gamma, \Omega)$;

\item $X$ satisfies $S_{1}(\Gamma_{cz}, \Omega)$;

\item $X$ satisfies $S_{1}(\Gamma_{F}, \Omega)$;

\item $C_p(X)$ satisfies $S_{1}(\Gamma_{\bf 0},\Omega_{\bf 0})$.

\end{enumerate}

\end{theorem}

\begin{proof}

$(1)\Rightarrow(2)$. Assume that $X$ is projectively
$S_{1}(\Gamma, \Omega)$. Let $(\mathcal{U}_n: n\in \mathbb{N})$ be
a sequence of countable covers of $X$ such that $\mathcal{U}_n\in
\Gamma_{cz}$ for each $n\in \mathbb{N}$. For every $n\in
\mathbb{N}$ and $U\in \mathcal{U}_n$, fix a continuous function
$f_U: X \mapsto \mathbb{R}$ such that
$U=f^{-1}_U(\mathbb{R}\setminus \{0\})$. Put $f=\prod\{f_U: U\in
\mathcal{U}_n, n\in \mathbb{N}\}$. Then $f$ is a continuous
mapping from $X$ to $\mathbb{R}^{\omega}$, and thus by (1),
$Y=f(X)$ has the property $S_{1}(\Gamma, \Omega)$. Put
$\mathcal{V}_n=\{f(U): U\in \mathcal{U}_n\}$. Then $\mathcal{V}_n$
is an $\gamma$-cover of $Y$. Since $Y$ has the property
$S_{1}(\Gamma, \Omega)$, there is $H_n\in \mathcal{V}_n$ such that
$\{H_n : n\in \mathbb{N}\}$ is $\omega$-cover of $Y$. Put
$F_n=f^{-1}(H_n)$. Then $F_n\in \mathcal{U}_n$, and $\{F_n : n\in
\mathbb{N}\}$ is $\omega$-cover of $X$.

$(2)\Rightarrow(1)$. Let $f:X\mapsto Y$ be continuous mapping from
$X$ onto a second countable space $Y$, and let $(\mathcal{U}_n:
n\in \mathbb{N})$ be a sequence of $\gamma$-covers of $Y$. Since
$Y$ is second countable, there is a countable subcover
$\mathcal{W}_n\subset \mathcal{V}_n$. Put
$\mathcal{O}_n=\{f^{-1}(W): W\in \mathcal{W}_n\}$. Then
$\mathcal{O}_n$ is a countable $\gamma$-cover of $X$ by cozero
sets. By (2), there is $H_n\in \mathcal{O}_n$ such that $\{H_n :
n\in \mathbb{N}\}$ is $\omega$-cover of $X$. For every $n\in
\mathbb{N}$, pick $U_H\in \mathcal{U}_n$ such that $U_H\supset
f(H)$. Put $\mathcal{F}=\{U_{H_n}: n\in \mathbb{N} \}$. Then
$\mathcal{F}$ is $\omega$-cover of $Y$. This proves that $Y$ has
the property $S_{1}(\Gamma, \Omega)$.

$(2)\Leftrightarrow(3)$. By Proposition 3.3 in \cite{sak1}.

$(3)\Leftrightarrow(4)$. By Proposition 6.4 in \cite{os2}.
\end{proof}

By Theorem 6.6 in \cite{os2} and Proposition \ref{pr80}, we have
the next result.

\begin{theorem}\label{th88} For a $z$-space $X$ with a coarser second countable topology, the following
statements are equivalent:

\begin{enumerate}

\item $C_p(X)$ satisfies $S_{1}(\mathcal{S},\mathcal{D})$;

\item $X$ satisfies $S_{1}(\Gamma_F, \Omega)$;

\item $C_p(X)$ satisfies $S_{1}(\Gamma_{\bf 0}, \Omega_{\bf 0})$;

\item $C_p(X)$ satisfies $S_{1}(\mathcal{S}, \Omega_{\bf 0})$;

\item $X$ is projectively $S_{1}(\Gamma, \Omega)$.

\end{enumerate}

\end{theorem}

\begin{proposition}
Every space of cardinality less than $\mathfrak{d}$ is
projectively $S_{1}(\Gamma,\Omega)$.
\end{proposition}

\section{The projectively $S_{fin}(\Gamma, \Omega)$ property}

\begin{proposition}\label{pr81} For a $z$-space $X$, the following statements are
equivalent:

\begin{enumerate}

\item $X$ is projectively $S_{fin}(\Gamma, \Omega)$;

\item $X$ satisfies $S_{fin}(\Gamma_{cz}, \Omega)$;

\item $X$ satisfies $S_{fin}(\Gamma_{F}, \Omega)$;

\item $C_p(X)$ satisfies $S_{fin}(\Gamma_{\bf 0},\Omega_{\bf 0})$.

\end{enumerate}

\end{proposition}

\begin{proof}
Similar the proof in Proposition \ref{pr80} and by Theorem 7.2 in
\cite{os2} and Theorem 76 in \cite{bcm}.
\end{proof}

By Theorem 7.2 in \cite{os2} and Proposition \ref{pr81}, we have
the next result.

\begin{theorem}\label{th888} For a $z$-space $X$ with a coarser second countable topology, the following
statements are equivalent:

\begin{enumerate}

\item $C_p(X)$ satisfies $S_{fin}(\mathcal{S},\mathcal{D})$;

\item $X$ satisfies $S_{fin}(\Gamma_F, \Omega)$;

\item $C_p(X)$ satisfies $S_{fin}(\Gamma_{\bf 0}, \Omega_{\bf
0})$;

\item $C_p(X)$ satisfies $S_{fin}(\mathcal{S}, \Omega_{\bf 0})$;

\item $X$ is projectively $S_{fin}(\Gamma, \Omega)$.

\end{enumerate}

\end{theorem}

\begin{proposition}
Every space of cardinality less than $\mathfrak{d}$ is
projectively $S_{fin}(\Gamma,\Omega)$.
\end{proposition}

\section{The projectively $S_{1}(\Gamma, \Gamma)$ property}

In \cite{sak1} (Theorem 2.5),  M. Sakai proved:

\begin{theorem} $(Sakai)$ \label{th151} For a space $X$, the following statements are
equivalent:

\begin{enumerate}

\item  $C_p(X)$ satisfies $S_{1}(\Gamma_{\bf 0}, \Gamma_{\bf 0})$;

\item $X$ satisfies $S_{1}(C_\Gamma, C_\Gamma)$ and it is strongly
zero-dimensional.

\end{enumerate}

\end{theorem}

\begin{theorem}(Theorem 67 in \cite{bcm})\label{th67}  The following properties
are equivalent for a space $X$:

\begin{enumerate}

\item  $C_p(X)$ satisfies $S_{1}(\Gamma_{\bf 0}, \Gamma_{\bf 0})$;

\item $X$ satisfies $S_{1}(\Gamma_{cz}, \Gamma)$;

\item $X$ is projectively $S_1(\Gamma, \Gamma)$.

\end{enumerate}

\end{theorem}

By Theorem 8.8 in \cite{os2} and Theorem \ref{th67}, we have the
next result.

\begin{theorem}\label{th158} For a $z$-space $X$ and $X$ $\models$ $V$, the following statements are equivalent:

\begin{enumerate}

\item $C_p(X)$ satisfies $S_{1}(\mathcal{S},\mathcal{S})$;

\item $X$ satisfies $S_{1}(\Gamma_F, \Gamma)$;

\item  $C_p(X)$ satisfies $S_{1}(\Gamma_{\bf 0}, \Gamma_{\bf 0})$;

\item  $C_p(X)$ satisfies $S_{1}(\mathcal{S}, \Gamma_{\bf 0})$;

\item $X$ is projectively $S_1(\Gamma, \Gamma)$;

\item $X$ is projectively $S_{fin}(\Gamma, \Gamma)$.

\end{enumerate}

\end{theorem}

\begin{proposition}
Every space of cardinality less than $\mathfrak{b}$ is
projectively $S_{1}(\Gamma,\Gamma)$.
\end{proposition}

\section{The projectively $U_{fin}(\mathcal{O}, \Omega)$ property}

By Theorem 3.4 in \cite{os10},  $X$ satisfies
$U_{fin}(\mathcal{O}, \Omega)$ if and only if $X$ satisfies
$U_{fin}(\Gamma_F, \Omega)$ and $X$ is Lindel$\ddot{o}$f.

\begin{theorem}\label{th422} For a space $X$, the following statements are
equivalent:

\begin{enumerate}

\item $X$ is projectively $U_{fin}(\mathcal{O}, \Omega)$;

\item $X$ satisfies $U_{fin}(\mathcal{O}^{\omega}_{cz}, \Omega)$;

\item $X$ satisfies $U_{fin}(\Gamma_F, \Omega)$.

\end{enumerate}

\end{theorem}

\begin{proof}

$(1)\Rightarrow(2)$. Assume that $X$ is projectively
$U_{fin}(\mathcal{O}, \Omega)$. Let $(\mathcal{U}_n: n\in
\mathbb{N})$ be a sequence of countable covers of $X$ by cozero
sets. For every $n\in \mathbb{N}$ and $U\in \mathcal{U}_n$, fix a
continuous function $f_U: X \mapsto \mathbb{R}$ such that
$U=f^{-1}_U(\mathbb{R}\setminus \{0\})$. Put $f=\prod\{f_U: U\in
\mathcal{U}_n, n\in \mathbb{N}\}$. Then $f$ is a continuous
mapping from $X$ to $\mathbb{R}^{\omega}$, and thus by (1),
$Y=f(X)$ has the property $U_{fin}(\mathcal{O}, \Omega)$. Put
$\mathcal{V}_n=\{f(U): U\in \mathcal{U}_n\}$. Then $\mathcal{V}_n$
is an open cover of $Y$. Since $Y$ has the property
$U_{fin}(\mathcal{O}, \Omega)$, there are finite subfamilies
$\mathcal{H}_n\subset \mathcal{V}_n$ such that $\bigcup
\{\mathcal{H}_n : n\in \mathbb{N}\}$ is $\omega$-cover of $Y$. Put
$\mathcal{F}_n=\{f^{-1}(H): H\in \mathcal{H}_n\}$. Then
$\mathcal{F}_n$ is a finite subfamily of $\mathcal{U}_n$, and
$\bigcup \{\mathcal{F}_n : n\in \mathbb{N}\}$ is $\omega$-cover of
$X$.

$(2)\Rightarrow(1)$. Let $f:X\mapsto Y$ be continuous mapping from
$X$ onto a second countable space $Y$, and let $(\mathcal{U}_n:
n\in \mathbb{N})$ be a sequence of open covers of $Y$. Since $Y$
is Lindel$\ddot{o}$f, there is a countable subcover
$\mathcal{W}_n\subset \mathcal{V}_n$. Put
$\mathcal{O}_n=\{f^{-1}(W): W\in \mathcal{W}_n\}$. Then
$\mathcal{O}_n$ is a countable cover of $X$ by cozero sets. By
(2), there are finite subfamilies $\mathcal{H}_n\subset
\mathcal{O}_n$ such that $\bigcup \{\mathcal{H}_n : n\in
\mathbb{N}\}$ is $\omega$-cover of $X$. For every $n\in
\mathbb{N}$ and every $H\in \mathcal{H}_n$, pick $U_H\in
\mathcal{U}_n$ such that $U_H\supset f(H)$. Put
$\mathcal{F}_n=\{U_H: H\in \mathcal{H}_n\}$. Then $\mathcal{F}_n$
is a finite subfamily of $\mathcal{U}_n$, and $\bigcup
\{\mathcal{F}_n : n\in \mathbb{N}\}$ is $\omega$-cover of $Y$.
This proves that $Y$ has the property $U_{fin}(\mathcal{O},
\Omega)$.

$(2)\Leftrightarrow(3)$. Proved analogously to the proof of
Theorem \ref{th22}.
\end{proof}

By Theorem \ref{th422}, Theorem 3.1 in \cite{os10} we have the
next result.

\begin{theorem}\label{th198} For a space $X$, the following statements are
equivalent:

\begin{enumerate}

\item $C_p(X)$ satisfies $F_{fin}(\Gamma_{\bf 0},\Omega_{\bf 0})$;

\item $X$ satisfies $U_{fin}(\Gamma_F, \Omega)$.

\item $X$ is projectively $U_{fin}(\mathcal{O}, \Omega)$;

\item $X$ satisfies $U_{fin}(\mathcal{O}^{\omega}_{cz}, \Omega)$.

\end{enumerate}

\end{theorem}

By Theorem \ref{th422}, Theorem 3.3 in \cite{os10} and Theorem
\ref{th198} we have the next result.

\begin{theorem}\label{th195} Let $X$ be a space with a coarser second countable topology. Then the following statements are equivalent:

\begin{enumerate}

\item $C_p(X)$ satisfies $S_{fin}(\mathcal{S},w\mathcal{D})$;

\item $X$ satisfies $U_{fin}(\Gamma_F, \Omega)$;

\item $C_p(X)$ satisfies $S_{fin}(\Gamma_{\bf 0}, w\Omega_{\bf
0})$;

\item $C_p(X)$ satisfies $S_{fin}(\mathcal{S}, w\Omega_{\bf 0})$;

\item $X$ is projectively $U_{fin}(\mathcal{O}, \Omega)$;

\item $X$ satisfies $U_{fin}(\mathcal{O}^{\omega}_{cz}, \Omega)$.

\end{enumerate}

\end{theorem}

\begin{proposition}
Every space of cardinality less than $\mathfrak{d}$ is
projectively $U_{fin}(\mathcal{O}, \Omega)$.
\end{proposition}

\section{The projectively $S_{1}(\Gamma, \mathcal{O})$ property}

Similarly to the proof  of Theorem \ref{pr80} we have the next
result.

\begin{proposition}\label{pr84} For a $z$-space $X$, the following statements are
equivalent:

\begin{enumerate}

\item $X$ is projectively $S_{1}(\Gamma, \mathcal{O})$;

\item $X$ satisfies $S_{1}(\Gamma_{cz}, \mathcal{O})$;

\item $X$ satisfies $S_{1}(\Gamma_{F}, \mathcal{O})$;

\item $C_p(X)$ satisfies $S_{1}(\Gamma_{\bf 0},\mathcal{D}_{\bf
0}[1])$.

\end{enumerate}

\end{proposition}

\begin{theorem}\label{th8888} For a $z$-space $X$ with a coarser second countable topology, the following
statements are equivalent:

\begin{enumerate}

\item $C_p(X)$ satisfies $S_{1}(\mathcal{S},\mathcal{D}[1])$;

\item $X$ satisfies $S_{1}(\Gamma_F, \mathcal{O})$;

\item $C_p(X)$ satisfies $S_{1}(\Gamma_{\bf 0}, \mathcal{D}_{\bf
0}[1])$;

\item $C_p(X)$ satisfies $S_{1}(\mathcal{S}, \mathcal{O}_{\bf
0})$;

\item $X$ is projectively $S_{1}(\Gamma, \mathcal{O})$.

\end{enumerate}

\end{theorem}

\medskip

\begin{proposition}
Every space of cardinality less than $\mathfrak{d}$ is
projectively $S_{1}(\Gamma, \mathcal{O})$.
\end{proposition}

We can summarize the relationships between considered notions in
next diagrams.

\begin{center}
\ingrw{90}{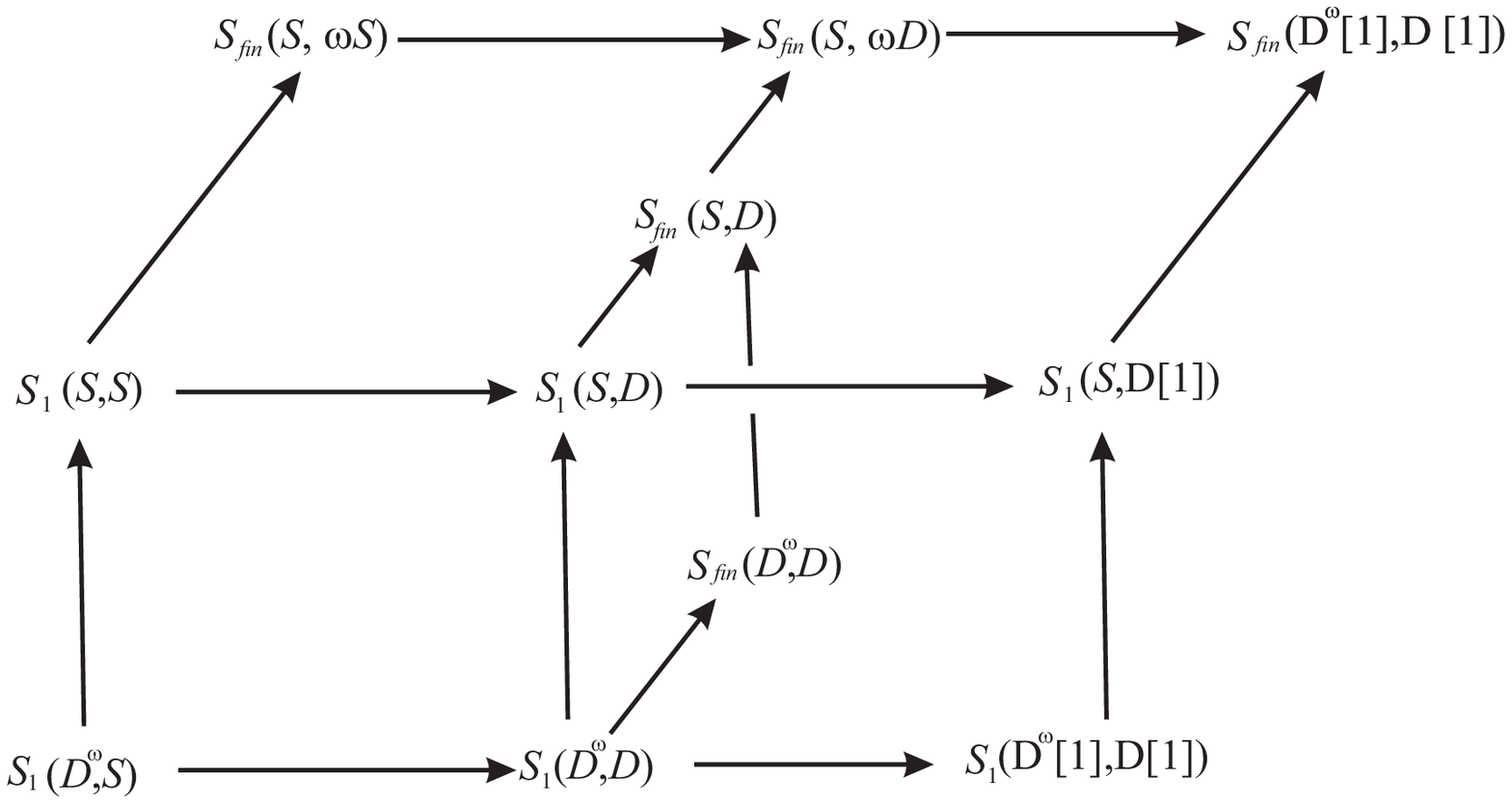}

\medskip

Fig.~2. The Diagram of selectors for sequences of dense
($1$-dense) sets of $C_p(X)$.

\end{center}

\bigskip

\begin{center}
\ingrw{90}{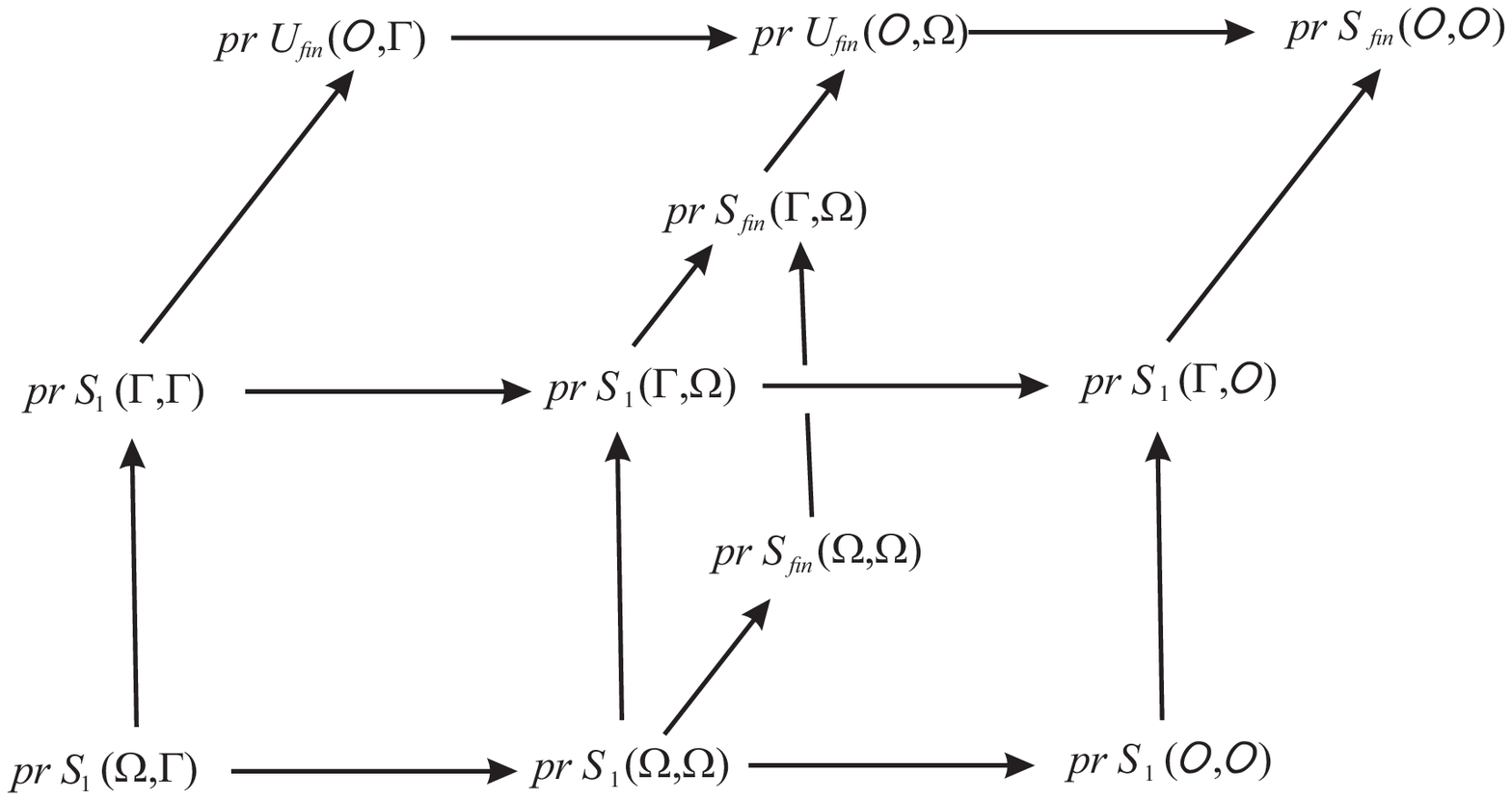}

Fig.~3. The Diagram of projective selection principles for a space
$X$ (with corresponding conditions) corresponding to selectors for
sequences of dense sets of $C_p(X)$.

\end{center}

\section{Acknowledgements}

The author would like to thank the referee for their careful
reading of this paper.

\bibliographystyle{model1a-num-names}
\bibliography{<your-bib-database>}







\end{document}